\documentclass{article}
\usepackage{amsmath, amssymb,verbatim, amsthm}
\usepackage{enumerate}
\usepackage{tikz-cd}
\usepackage{pst-node}
\usepackage[nottoc]{tocbibind}
\usepackage{qtree}

\usepackage{tikz}
\usetikzlibrary{automata, positioning, arrows}

\tikzset{
        ->, 
        >=stealth,
        node distance=3cm,
        every state/.style={thick, fill=gray!10}, 
        initial text=$ $, 
        }

\usepackage{pgfplots}

\usepackage{mathrsfs}

\pgfplotsset{compat = newest}
\usepackage{graphicx}
\usepackage{caption}
\usepackage{subcaption}
\graphicspath{ {./images/} }
\newcounter{exercise}
\newcommand{\DL}{{\rm DL}}

\newcommand{\Isom}{{\rm Isom}}
\newcommand{\Aut}{{\rm Aut}}

\newcommand{\set}[1]{\left\{#1\right\}}

\newcommand{\F}{\mathbb{F}}

\newcommand{\bbH}{\mathbb{H}}

\newcommand{\Z}{\mathbb{Z}}

\newcommand{\fee}{\varphi}

\usepackage{hyperref}
\hypersetup{
    colorlinks,
    citecolor=black,
    filecolor=black,
    linkcolor=black,
    urlcolor=black
}

\title{Structure of Cross-wired Lamplighter Groups}
\date{\today}
\author{Benjamin Jeffers}

\newtheorem{theorem}{Theorem}
\newtheorem{prop}[theorem]{Proposition}
\newtheorem{lemma}[theorem]{Lemma}

\theoremstyle{definition}
\newtheorem{definition}[theorem]{Definition}

\newtheorem*{remark}{Remark}

\begin{document}

\maketitle

\abstract{We answer a question posed in \cite{CWL}, proving that for a closed cocompact subgroup $\Gamma$ of $\Isom(\DL(n, n))$ not contained in $\Isom^+(\DL(n, n))$, the sequence $1 \to H \to \Gamma \to D_\infty \to 1$ splits, where $H$ is the unique open normal subgroup such that $\Gamma/H \cong D_\infty$.}


\section{Introduction}
Cross-wired lamplighter groups were introduced in \cite{CWL} as a generalization of lamplighter groups. Traditional lamplighters are a heavily studied object, it is defined as the restricted wreath product $\Z_2 \wr \Z = \left( \bigoplus_{\Z} \Z/2\Z \right) \rtimes \Z$ where the action of $\Z$ on $\bigoplus_{\Z} \Z/2\Z$ is given by shifting the coordinates in the direct sum. Geometrically, the lamplighter group can be thought of as an infinite street with an infinite number of lamps, and a person to light the lamps, called the lamplighter. The lamplighter can move left and right, turning the lamps on or off, while only a finite number of lamps are on. An element in $\Z_2 \wr \Z$ corresponds to the positions of the on lamps, and the integer $n \in \Z$ is the position of the lamplighter. More generally, for any finite group $F$, the lamplighter group of $F$ is the wreath product $F \wr \Z = \left(\bigoplus_\Z F \right)\rtimes \Z$. This group is step-2 solvable and it has exponential growth. Grigorchuk and Zuk showed that $\Z/2\Z \wr \Z$ can be constructed as the automata groupof a 2-state automaton. Yang further to show that many automata groups arising from the Cayley machines defined by Steinberg and Silva, \cite{Steinberg-Silva}, are also cross-wired lamplighters, \cite{Yang1}

\par For one particular choice of generators the Cayley graph of a lamplighter group, is a Diestel-Leader graph \cite{DL}. For two positive integers, $m, n \geq 2$, the vertex set of a Diestel-Leader graph is defined as \[
DL(m, n) = \set{(x, y) \in T_m \times T_n \mid b(x) + b'(y) = 0}
\] where $T_m$ and $T_n$ are the regular infinite tree of degrees $m$ and $n$ respectively, and $b, b'$ are a Busemann functions on $T_m$ and $T_n$ respectively. Two vertices $(x, y)$, $(x', y')$ are connected by an edge if $x$ and $x'$ are connected by an edge in $T_m$ and $y$ and $y'$ are connected by an edge in $T_n$. Eskin, Fisher, and Whyte proved that a finitely generated group is quasi-isometric to a lamplighter group $F \wr \Z$ if and only if it acts properly and cocompactly by isometries on a Diestel-Leader graph $DL(n, n)$ \cite{EFW1}. Moreover, they showed that if $m \neq n$, then $DL(m, n)$ is not quasi-isometric to the Cayely graph of any finitely generated group, proving a conjecture of Diestel and Leader, \cite{EFW2}. 

In \cite{CWL}, Cornulier, Fisher, and Kashyap study the quasi-isometric rigidity of lamplighter groups. They do this by studying the cocompact lattices in the isometry group of a Diestel-Leader graph, $\Isom(\DL(n, n))$, and closed, cocompact subgroups of $\Isom(\DL(m, n))$ when $m \neq n$. These lattices turn out to not be lamplighter groups in general, and instead are called \emph{cross-wired lamplighters}. The name is due to the fact that we can view them as an infinite street with an infinite number of lamps and a lamplighter. The difference is that when the lamplighter changes the state of one lamp, it might change the state of other lamps as well. 

\par In this paper we answer question 1.4(iii) posed in \cite{CWL}. They ask, if given $\Gamma$, a closed cocompact subgroup of $\Isom(\DL(n, n))$ not contained in $\Isom^+(\DL(n, n))$ , does the extension \[
1 \to H \to \Gamma \to D_{\infty} \to 1
\] necessarily split? Here $H$ is the unique open normal subgroup such that $\Gamma/H \cong D_\infty$, and $D_{\infty}$ is the infinite dihedral group. We prove that the answer is yes by giving an isomorphism between the groups $(U \rtimes_\phi \Z) \rtimes_\psi Z/2\Z$ and $U \rtimes_\phi (\Z \rtimes_\psi \Z/2\Z)$. At the end we make a couple of comments on question 1.4(i) posed in \cite{CWL}. They ask if there exists a cross-wired lamplighter group that is not virtually symmetric. Where a cross-wired lamplighter $\Gamma$ is symmetric if it admits an automorphism $\alpha$ such that $\pi \circ \alpha = -\pi$ where $\pi$ is the projection from $\Gamma$ to $\Z$. A cross-wired lamplighter is virtually symmetric if it has a finite index symmetric subgroup. While we do not prove the statement either way we demonstrate that the example given in \cite{CWL}  
of $\Gamma = \mathbb H(\F_q[X, X^{-1}]$ is symmetric, where $\Gamma$ is the Heisenberg group of upper triangular matrices over $\F_q[X, X^{-1}]$. 

\par
\vspace{.5cm}
\noindent {\bf Acknowledgments} I would like to thank Nata\v sa Macura for the suggestion of the problem and many helpful discussions and Ryan Daileda for many fruitful conversations.


\section{Preliminaries}

In this section we present a general theorem that will be of use to us, as well as construct the Diestel-Leader graph, and describe its isometry group, and then we end with a brief discussion of cross-wired lamplighers including the main structure theorem of \cite{CWL}. Throughout we denote the regular infinite tree of degree $n$ by $T_n$. 

We will make use of the following theorem characterizing splittings of short exact sequences of groups to prove our main theorem. 

\begin{definition}\label{split-lemma}
We say that a sequence of groups $1 \to H \xrightarrow{\alpha} G \xrightarrow{\beta} K \to 1$ \emph{splits} if there is a homomorphism $\varphi \colon K \to \Aut(H)$ and an 
    \[ \begin{tikzcd}
1 \arrow{r} & H \arrow{d}{\text{id}} \arrow{r}{\alpha} & G \arrow{d}{\theta} \arrow{r}{\beta} & K \arrow{d}{\text{id}} \arrow{r}& 1 \\%
1 \arrow{r} & H \arrow{r} & H \rtimes_{\fee} K \arrow{r} & K \arrow{r} & 1
\end{tikzcd}
\] commutes. 
\end{definition}

\begin{definition}
The infinite dihedral group, denoted $D_\infty$ is given by the presentation: 
\[
D_{\infty} \cong \Z \rtimes \Z/2\Z = \left< r, s \mid s^2 = 1, srs = r^{-1} \right>
\]
\end{definition}

\begin{remark}
The group $D_\infty$ can be thought of as the isometries of $\Z$. It is generated by a shift to the right by one, and a flip around the origin.
\end{remark}

\vspace{.5cm}

Next, we construct the Diestel-Leader graph as in \cite{DL}, except we only consider the case when $m = n$. Given a $n$-regular infinite tree, fix a preferred end $-\infty$. For any two vertices we define their \textit{confluent} $\widehat{xy}$ to be the point where the infinite ray from $x$ to $-\infty$ meets the infinite ray from $y$ to $-\infty$. Now, fix a vertex $o$, this will be the origin. Define the \textit{Busemann function} or \textit{height function} to be $b \colon V(T_n) \to \Z$ given by \[
b(x) = d(x, \widehat{xo}) - d(o, \widehat{xo}).
\] This will partition the vertices of $T_n$ into sets which we call the \textit{level sets} $H_k$. An example of the level sets with respect to an origin and $-\infty$ in $T_2$ can be seen in figure \ref{fig:tree}. To define the Diestel-Leader graph, we take two copies of the tree $T_n$ along with a Busemann functions, $b$ and let \[
V(\DL(n, n)) = \set{(x, y) \in V(T_n \times T_n) \mid b(x) + b(y) = 0}. 
\] Two vertices $(a, b), (c, d)$ are connected in $\DL(n, n)$ if and only if $a$ is adjacent to $c$ in $T_n$ and $b$ is adjacent to $d$ in the second copy of $T_n$.

\begin{remark}
The Diestel-Leader graph is often depicted with the two trees side by side, with one upside down, so that the level sets add up to zero, as can be seen in figure \ref{fig:two-trees}. 
\end{remark}

\begin{figure}[h!]
\centering
\includegraphics[width=12cm]{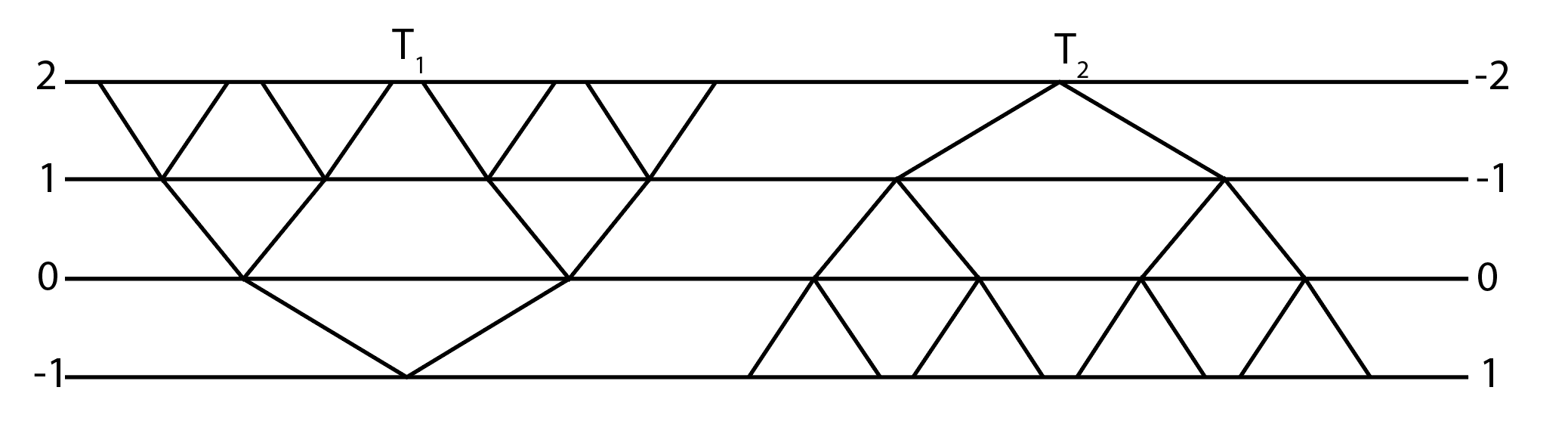}
\caption{The trees $T_1$ and $T_2$ side by side with the height functions of the vertices.}
\label{fig:two-trees}
\end{figure}

\begin{remark}
The resulting tree and level sets is independent of the choices of $-\infty$ and $o$ since a a tree with different choice of origin and boundary point can be interchanged by an isometry
\end{remark}

\begin{figure}[h!]
    \centering
    \includegraphics[width=17cm]{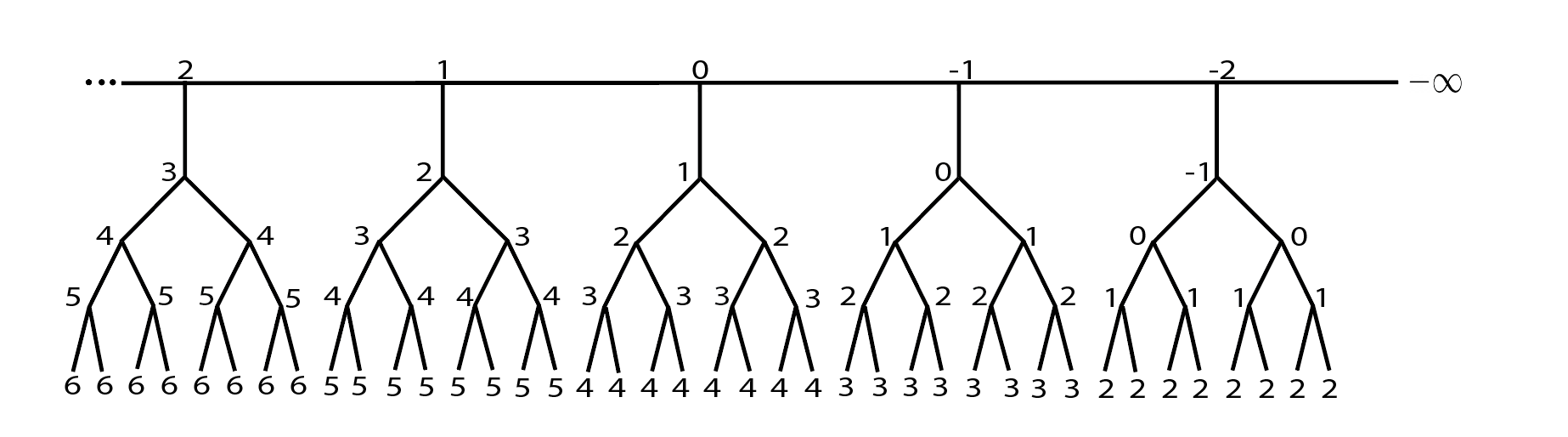}
    \caption{Level sets of $T_2$}
    \label{fig:tree}
\end{figure}

Now we construct the group of isometries of $\DL(n, n)$. See  \cite{WoessDL} for more details. Let $U_0$ denote the isometries of $T_n$ which fix $-\infty$ and preserve the level sets of the Busemann function $b$. We can explicitly describe this group. Take a vertex $v$, in a level set $H_k$, that is on the main line with the origin $o$ as in figure \ref{fig:tree}. The children of $v$ will be the adjacent vertices in level set $H_{k+1}$. An automorphism that fixes $-\infty$ and preserves the level set will be a rooted tree automorphisms with $v$ as the root. We can compose also compose these automorphisms with other rooted tree automorphisms taking different vertices on the main line as a root. These compositions will also preserve $-\infty$ and the level sets. In fact, these are all the possible automorphisms in $U_0$ as is shown in the next lemma. 

\begin{lemma}
The only isometries in $U_0$ are those that are a composition of rooted tree automorphisms to each branch, as well as the isometries that swap the children of a vertex on the main branch.  
\end{lemma}

\begin{proof}
It is well known that the only other types of infinite tree isometries are those that push everything along one line in the tree, and those that rotate around a vertex. If we rotate around a specific vertex, then this will not preserve $-\infty$, thus they are not in $U_0$. 

Let $A$ be an axis in the tree, and let $x$ be a vertex on $A$ such that $b(x) = k$. Then when we push along $A$, the vertex $x$ must end up in the level set $H_k$, however, looking at the child of $x$ that also lies in $A$, which we will call $y$, we see that this vertex must end up as a parent of $x$ in the final position, meaning $b(y) = k-1$, thus this isometry does not preserve the level sets. 
\end{proof}

Then the group $U = U_0 \times U_0$ is contained in $\Isom(\DL(n, n))$. Let $\varphi$ be the hyperbolic isometry of $T_n$ that has translation length 1, with $-\infty$ as the repelling end. This means that $\varphi$ has no fixed points, and we can essentially think of it as moving branch 0 to branch 1, branch 1 to branch 2 and so forth as seen in figure \ref{fig:action}\begin{figure}[h!].
    \centering
    \includegraphics[width=15cm]{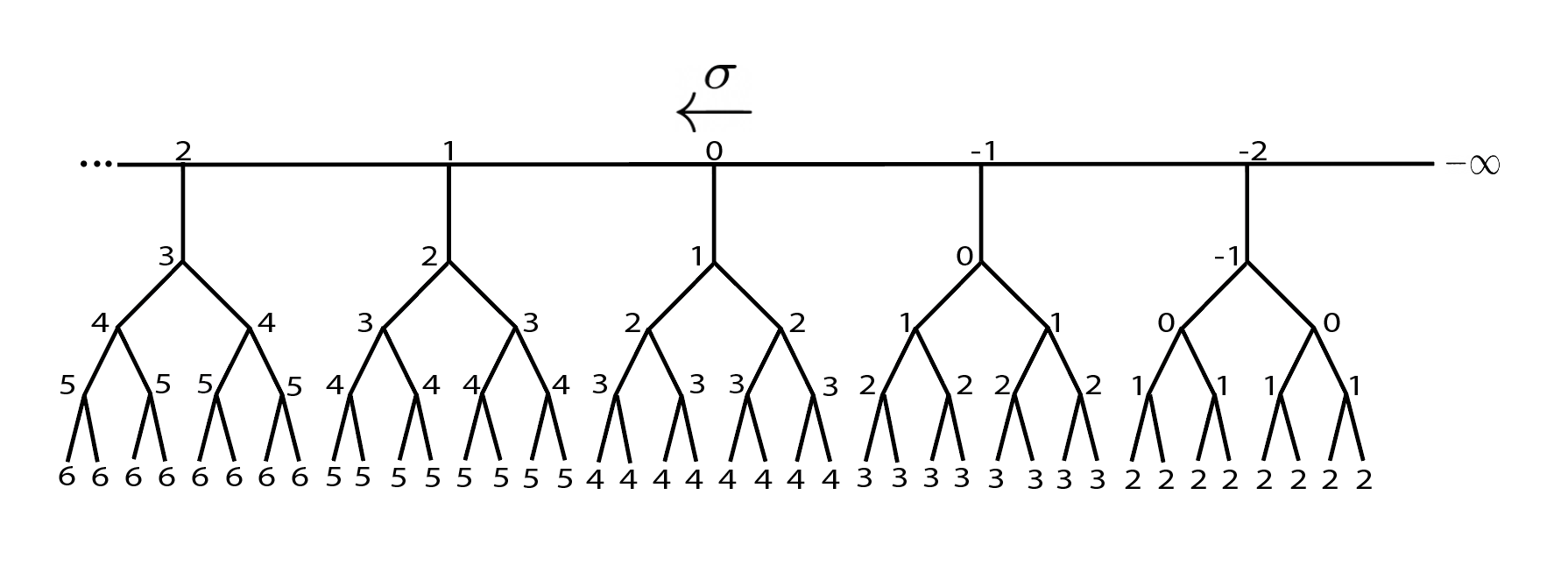}
    \caption{The action of $\phi_2$ on $T_2$}
    \label{fig:action}
\end{figure} 

Define the isometry $(x, y) \mapsto (\phi(x), \phi^{-1}(y))$ of $T_n\times T_n$. This shifts one of the trees up a level, so a vertex of $T_n$ with height $i$ will now have height $i+1$, while a vertex of the other copy of $T_n$ with height $-i$ will now have height $-i - 1$, so that the height still sums to 0. Finally, there exists an isometry $\psi$ defined by $\psi(x, y) = (y, x)$, it swaps the two trees.  
\par A result by Bartholdi, Neuhauser, and Woess in \cite{BNW} shows that $\Isom(\DL(n, n)$ contains the groups generated by $U, \phi$, and $\psi$, and nothing else.
\begin{prop}
For $n \geq 2$, \[
\Isom(\DL(n, n)) = (U \rtimes_\phi \Z) \rtimes_\psi \Z/2\Z). 
\]
\end{prop}


Next we take a brief look at the actions used to define the isometry group. Note that an element of the group comes in the form $(((u_0, u_1), a), \epsilon)$ where $u_0, u_1$ are level set preserving and $-\infty$ fixing isometries of $T_n$, $a$ is an integer and $\epsilon \in \Z/2\Z$.  

\par The action defined by $\phi$ is $\phi \colon \Z \to \Aut(U)$ given by $\phi(a)(u_0, u_1) = \phi_a(u_0, u_1) = (\phi^a u_0\phi^{-a}, \phi^{-a}u_1 \phi^a)$. We also have an isometry where we flip the two trees, $(x, y) \mapsto (y, x)$. Call this flip $\psi$. This action is given by $\psi \colon \Z /2\Z \to \Aut(U \rtimes_{\phi} \Z)$ where $\psi(\epsilon)(u_0, u_1) = \psi_{\epsilon}((u_0, u_1), a) = ((u_{0 + \epsilon}, u_{1 + \epsilon}), (-1)^{\epsilon}a)$ for $\epsilon \in \Z/2\Z$. Where $u_{1 + \epsilon}$ is taken to be modulo 2, this way if $\epsilon = 1$, then we are flipping the two isometries. 

\vspace{0.5cm} Looking at the multiplication in the group $(U \rtimes_{\phi} \Z) \rtimes_{\psi} \Z/2\Z$, we have \begin{align*}
    (((u_0, u_1), a), \epsilon) \cdot (((v_0, v_1), b), \delta) =&\; (((u_0, u_1), a)\cdot \psi_{\epsilon}((v_0, v_1), b), \epsilon + \delta)\\
    =&\; (((u_0, u_1), a)\cdot ((v_{0 + \epsilon}, v_{1 + \epsilon}), (-1)^{\epsilon}b), \epsilon + \delta)\\
    =&\; (((u_0, u_1)\cdot \phi_a(v_0, v_1)), a + (-1)^{\epsilon}b), \epsilon + \delta)\\
    =&\; ((u_0\phi^av_{0 + \epsilon}\phi^{-a}, u_1\phi^av_{1 + \epsilon}\phi^{-a}), a + (-1)^{\epsilon} \cdot b), \epsilon + \delta).
\end{align*}


Geometrically, an element $(((u_0, u_1), a), \epsilon) \in (U \rtimes_{\phi} \Z) \rtimes_{\psi} \Z/2\Z$ contains $u_0, u_1$, the two level preserving isometries to perform on the two trees that make up the Diestel-Leader graph, and an integer, $a$, which tells us how many times we should perform $\phi$ on the first tree and how many times we should perform $\phi^{-1}$ on the second tree, and finally $\epsilon$ tells us whether or not we should flip the two trees. When we multiply elements and get the conjugation action on the isometries $(u_0, u_1)$ this moves the isometries up or down the number of times we perform $\phi$.

\begin{definition}
Let $\Isom(\DL(m, n))$ be the isometry group of $\DL(m ,n)$. An isometry $f \in \Isom(\DL(n, n))$ is \emph{positive} if it is the restriction to $\DL(n, n)$ of some product isometry $(\alpha, \beta) \in \Aut(T_n \times T_n)$ of $T_n \times T_n$. A non positive isometry is the composition of a positive isometry and a flip $(x, y) \mapsto (y, x)$. Note that this can only happen when $m =n$, otherwise we could not perform a flip since the two trees making up the Diestel-Leader graph will have different valences. 
\end{definition}

We let $G = \Isom^+(\DL(n, n)) = U \rtimes_\phi \Z$ be the index two subgroup of $\Isom(\DL(n, n))$. Let $\pi \colon G \to \Z$ be the projection onto $\Z$ so that $U = \ker \pi$. 

The following is corollary 2.3 from \cite{CWL}.
\begin{lemma}
If $\Gamma \leq \Isom(DL(n, n))$ is a closed cocompact subgroup,then there is an element $t \in \Gamma$ such that $\Gamma = (\Gamma \cap U) \rtimes \left< t\right>.$ 
\end{lemma}

Finally we end this statement with the main structure theorem from \cite{CWL}.

\begin{theorem}\label{main-thm-cwl}
Suppose $n \geq 2$. 
\begin{enumerate}
\item Let $\Gamma$ be a closed, cocompact subgroup of $\Isom^+(\DL(n, n))$. Then $\Gamma$ has a unique open normal subgroup $H$ such that $\Gamma /H$ is infinite cyclic. Moreover, if $t$ is any element of $\Gamma$ mapping to a generator of $\Gamma/H$, then $H$ has two open subgroups $L, L'$ such that:
\begin{itemize}
\item $tLt^{-1}$ and $t^{-1}L't$ are subgroups of index $n$ in $L$ and $L'$ respectively.
\item $\bigcup_{k \in \Z} t^{-k}Kt^k = \bigcup_{k \in \Z}t^kL't^{-k} = H$, the unions are increasing.
\item $L \cap L'$ is a vertex stabilizer and thus is compact.
\item $LL' = H$.
\end{itemize} Moreover, $\Gamma$ has no nontriial compact normal subgroup and $\Gamma$ is discrete. If the action of $\Gamma$ is simply transitive, then $L \cap L' = \set{1}$ and $LL' = H$. 
\item Conversely, let $\Gamma$ be a locally compact group, with a semidirect product decomposition $\Gamma = H \rtimes \left< t \right>$, with $H$ noncompact. Assume that $H$ has open subgroups $L, L'$ such that: 
\begin{itemize}
\item $tLt^{-1}$ and $t^{-1}Lt$ are finite index subgroups, of index $n$, in $L$ and $L'$ respectively.
\item $\bigcup_{k \in \Z} t^{-k}Lt^k = \bigcup_{k \in \Z} t^kL't^{-k} = H$ (the unions are increasing).
\item $L \cap L'$ is compact.
\item the double coset space $L\backslash H/L'$ is finite of cardinality $d$.
\end{itemize}
Then $H$ is locally elliptic (i.e., every compact subset of $H$ is containd in a compact subgroup), and $\Gamma$ has a proper, transitive action on $\DL(m, n)$, for which $L \cap L'$ is a vertex stabilizer, and whose kernel is $\bigcap_{k \in \Z} t^{-k}(L \cap L')t^k$. Moreover, if $\Gamma$ is discrete, then it is finitely generated. 
\end{enumerate}
\end{theorem}

\begin{remark}
Given a cross-wired lampligher $\Gamma$ contained in $\Isom(\DL(n, n))$, and not contained in $\Isom^+(\DL(n, n))$, we can take the index two subgroup of $\Gamma$, call it $\Gamma'$, that is contained in $\Isom^+(\DL(n, n))$, and by Theorem \ref{main-thm-cwl}, there is a unique subgroup $H$ such that $\Gamma'/H \cong \Z$. Thus $\Gamma/H \cong D_{\infty}$. 
\end{remark}


\section{Proof of the Main Theorem}

\vspace{0.5cm}
Now, we look at the action in the group $U \rtimes_{\phi} (\Z \rtimes_{\psi} \Z/2/\Z)$. We first see that the action of $\Z/2\Z$ on $\Z$ is defined in the only possible way, by sending $\epsilon$ to the automorphism $(-1)^{\epsilon}\cdot n$ for $n \in \Z$. 

\par We define the action of $U \rtimes_{\phi} ( \Z \rtimes_{\psi} \Z/2\Z)$ by the homomorphism, $\phi \colon \Z \rtimes_{\phi} \Z/2\Z \to \Aut(U)$ given by $\phi(a, \epsilon). =\phi_{(a, \epsilon)}(u_0. u_1) = (\phi^a u_{0 + \epsilon}\phi^{-a}, \phi^{-a}u_{1 + \epsilon} \phi^a)$ for $(a, \epsilon) \in \Z \rtimes_{\phi} \Z/2\Z$. We first check that it is a homomorphism: \begin{align*}
    \phi(a, \epsilon)\phi(b, \delta)(u_0, u_1) =&\; \phi(a, \epsilon)(\phi^b u_{0 + \delta} \phi^{-b}, \phi^{-b}u_{1 + \delta}\phi^b)\\
    =&\; (\phi^a\phi^{(-1)^{\epsilon}b}u_{0 + \epsilon + \delta}\phi^{(-1)^{\epsilon + 1}b}\phi^{-a}, \phi^{-a}\phi^{(-1)^{\epsilon + 1}b} u_{1 + \epsilon + \delta} \phi^{(-1)^{\epsilon}b}\phi^a)\\ 
    =&\; (\phi^{1 + (-1)^{\epsilon}b}u_{0+ \epsilon + \delta} \phi^{-a -(-1)^{\epsilon}b}, \phi^{-a -(-1)^{\epsilon}b}u_{1 + \epsilon + \delta}\phi^{1 + (-1)^{\epsilon}b}) \\
    =&\;  \phi(a + (-1)^{\epsilon}b, \epsilon + \delta) = \phi((a , \epsilon)(b, \delta))
\end{align*}

Next, we check that it is an automorphism. Let $(u_0, u_1), (v_0, v_1) \in U$ and $(a, \epsilon) \in \Z \rtimes_{\phi} \Z/2\Z$. Then \begin{align*}
    \phi_{(a, \epsilon)}(u_0, u_1)\phi_{(a, \epsilon)}(v_0, v_1) =&\; (\phi^au_{0 + \epsilon}\phi^{-a}, \phi^{-a}u_{1 + \epsilon}\phi^a)(\phi^av_{0 + \epsilon}\phi^{-a}, \phi^{-a}v_{1 + \epsilon}\phi^a)\\
    =&\; (\phi^au_{0 + \epsilon}v_{0 + \epsilon}\phi^{-a}, \phi^{-a}u_{1 + \epsilon}v_{1 + \epsilon}\phi^a)
\end{align*} and \[
\phi_{(a, \epsilon)}(u_0v_0, u_1v_1) = (\phi^au_{0 + \epsilon}v_{0 + \epsilon}\phi^{-a}, \phi^{-a}u_{1 + \epsilon}v_{1 + \epsilon}\phi^a)
\] It is also clear that this map is bijective,and thus an isomorphism.

\vspace{0.5cm}
Then looking at the multiplication in the group we have \[
((u_0, u_1), (a, \epsilon)) \cdot ((v_0, v_1), (b, \delta)) = ((u_0v_{0 + \epsilon}\phi^a, u_1v_{1 + \epsilon}\phi^{-a}), (a + (-1)^{\epsilon} \cdot b, \epsilon + \delta)).
\] Since the action gives us the same triple of elements after multiplication, we can conclude they are isomorphic by sending $(((u_0, u_1), a), \epsilon)$ and $((u_0, u_1), (a, \epsilon))$ to $((u_0, u_1), a, \epsilon)$. 

Combining all of the above we get the following lemma.

\begin{lemma}\label{main-lem}
$(U \rtimes_{\phi} \Z) \rtimes_{\psi} \Z/2\Z \cong U \rtimes_{\phi} (\Z \rtimes_{\psi} \Z/2\Z)$. 
\end{lemma}

Now we turn to proving the main theorem we are interested in by first stating a lemma that is an extension of corollary 2.3 from \cite{CWL}.

\begin{lemma}
Let $\Gamma \leq \Isom(\DL(n, n))$ be a closed, cocompact subgroup that is not contained in $\Isom^+(\DL(n, n))$. Then there is an element $t \in \Gamma$ such that $\Gamma = ((\Gamma \cap U) \rtimes \langle t \rangle) \rtimes_{\psi} \Z/2\Z$. 
\end{lemma}

\begin{proof}
First take $\Gamma' = \Gamma \cap \Isom^+(\DL(n, n))$, then by corollary 2.3 from \cite{CWL}, there is an element $t \in \Gamma'$ such that $\Gamma' = (\Gamma' \cap U) \rtimes \left< t \right>$. Finally, since $\Isom^+(\DL(n, n))$ is an index 2 subgroup in $\Isom(\DL(n, n))$, we see that $\Gamma'$ is an index 2 subgroup in $\Gamma$, and thus $\Gamma = \Gamma' \rtimes \Z/2\Z = (\Gamma' \cap U) \rtimes \left< t \right>$. But $\Gamma' \cap U = \Gamma \cap U$, giving us the result. 
\end{proof}

Following the proof of the theorem 1.1 \cite{CWL}, we let $H = \Gamma \cap U$. Then $\Gamma/H \cong D_{\infty}$. For clarity, we will write $\Gamma$ as $(H \rtimes \Z )\rtimes \Z/2\Z$. Now, we can rewrite our initial short exact sequence as \[
1 \to H \to (H \rtimes \Z )\rtimes \Z/2\Z \to \Z \rtimes \Z/2\Z \to 1.
\] However, by lemma \ref{main-lem} we get that $((\Gamma \cap U) \rtimes \Z )\rtimes \Z/2\Z \cong (\Gamma \cap U) \rtimes (\Z \rtimes \Z/2\Z)$. Lemma \ref{main-lem} holds for subgroups of $U$ since we can just map both $((\Gamma \cap U) \rtimes \Z )\rtimes \Z/2\Z$ and $(\Gamma \cap U) \rtimes (\Z \rtimes \Z/2\Z)$ to $(\Gamma \cap U) \rtimes_{\phi} \Z \rtimes_{\psi} \Z/2\Z$ the same way we did above.  So, replacing $H = \Gamma \cap U$ we get \[
1 \to H \to H \rtimes (\Z \rtimes \Z/2\Z) \to \Z \rtimes \Z/2\Z \to 1.
\] 
which by lemma \ref{split-lemma}, tells us that our exact sequence splits. Thus our extension of interest splits and we get $\Gamma \cong H \rtimes D_{\infty}$.

\section{Symmetry}

We make a few remarks about the symmetry question in \cite{CWL}. 

\begin{definition}
Let $\Gamma$ be a cross-wired lamplighter group. We say $\Gamma$ is symmetric if $\Gamma$ admits an automorphism $\alpha$ such that $\pi \circ \alpha = -\pi$ where $\pi \colon \Gamma = H \rtimes_\phi \Z \to \Z$ is the projection onto the second coordinate. We say that $\Gamma$ is \emph{virtually symmetric} if $\Gamma$ has a symmetric subgroup of finite index. 
\end{definition}

In \cite{CWL}, Cournulier, Fisher and Kashyap ask whether or not there exists a non virtually-symmetric cross-wired lamplighter. 

We provide an example showing that the cross-wired lamplighter group $\bbH(\F_q[X, X^{-1}]) \rtimes_\phi \Z$ is symmetric where $\bbH(\cdot)$ is the Heisenberg group of upper triangular matrices over a ring.

\begin{prop}
The group $\bbH(\F_q[X, X^{-1}]) \rtimes_{\phi} \Z$ is symmetric.
\end{prop}

\begin{proof}
Let $\\sigma \colon \Gamma \to \Gamma$ be given by \[
\sigma\left(\begin{pmatrix}
1 & f(X, X^{-1}) & h(X, X^{-1}) \\
0 & 1 & g(X, X^{-1})\\
0 & 0 & 1
\end{pmatrix},a
\right) = \left(\begin{pmatrix}
1 & f(X^{-1}, X) & h(X^{-1}, X)\\
0 & 1 & g(X^{-1}, X)\\
0 & 0 & 1
\end{pmatrix}, -a\right).
\] So $\sigma$ restricts to an automorphism of $\F_q[X, X^{-1}]$ that sends $X$ to $X^{-1}$. It is clear the $\sigma$ is bijective. We check that it is a homomorphism: \begin{align*}
    &\sigma \left(\begin{pmatrix}
    1 & f(X, X^{-1}) & h(X, X^{-1}) \\
    0 & 1 & g(X, X^{-1})\\
    0 & 0 & 1
    \end{pmatrix}, a\right) \sigma\left(\begin{pmatrix}
    1 & i(X, X^{-1}) & k(X, X^{-1}) \\
    0 & 1 & j(X, X^{-1})\\
    0 & 0 & 1
    \end{pmatrix} ,b\right)\\ =& \; \left(\begin{pmatrix}
1 & f(X^{-1}, X) & h(X^{-1}, X)\\
0 & 1 & g(X^{-1}, X)\\
0 & 0 & 1
\end{pmatrix}, -a\right) \left(\begin{pmatrix}
1 & i(X^{-1}, X) & k(X^{-1}, X)\\
0 & 1 & j(X^{-1}, X)\\
0 & 0 & 1
\end{pmatrix}, -b\right)\\
=&\; \left(\begin{pmatrix}
1 & f(X^{-1} X) + X^{-a}i(X^{-1}, X) & X^{-2a}h(X^{-1}, X) + X^{-a}f(X^{-1}, X)j(X^{-1}, X) + h(X^{-1}, X)\\
0 & 1 & g(X^{-1}, X) + X^{-a}j(X^{-1}, X)\\
0 & 0 & 1
\end{pmatrix}, -a-b\right)
\end{align*} and \begin{align*}
    &\sigma\left(\left(\begin{pmatrix}
    1 & f(X, X^{-1}) & h(X, X^{-1}) \\
    0 & 1 & g(X, X^{-1})\\
    0 & 0 & 1
    \end{pmatrix}, a\right)\left( \begin{pmatrix}
    1 & i(X, X^{-1}) & k(X, X^{-1}) \\
    0 & 1 & j(X, X^{-1})\\
    0 & 0 & 1
    \end{pmatrix}, b\right)\right)\\
   =& \sigma \left(\begin{pmatrix}
    1 & f(X, X^{-1}) + X^ai(X, X^{-1}) & x^{2a}k(X, X^{-1}) + X^af(X, X^{-1})j(X, X^{-1}) + h(X, X^{-1})\\
    0 & 1 & g(X, X^{-1}) + X^aj(X, X^{-1})\\
    0 & 0 & 1
    \end{pmatrix}, a + b\right)\\
    =& \; \left(\begin{pmatrix}
    1 & f(X^{-1} X) + X^{-a}i(X^{-1}, X) & X^{-2a}h(X^{-1}, X) + X^{-a}f(X^{-1}, X)j(X^{-1}, X) + h(X^{-1}, X)\\
0 & 1 & g(X^{-1}, X) + X^{-a}j(X^{-1}, X)\\
0 & 0 & 1
    \end{pmatrix}\right)
\end{align*} Thus $\sigma$ is an automorphism and $\Gamma$ is symmetric. 
\end{proof}

We end by stating a few open questions about cross-wired lamplighters.
\begin{enumerate}
\item Is there a cross-wired lamplighter that is not virtually symmetrc?
\item Given a cross-wired lamplighter, is the pair $\set{L, L'}$ unique up to commensurability and automorphisms? Suppose that there are two embeddings of $\Gamma$ as a cross-wired lamplighter group $(L_1, L_1')$ and $(L_2, L_2')$. Is there an automorphism $\beta$ of $\Gamma$ such that $\beta(L_1)$ is commensurable to $L_1'$ and $\beta(L_2)$ is commensurable to $L_2'$. This is question 1.4(ii) from \cite{CWL}.
\item Is there a non-linear cross-wired lamplighter? All of the example of cross-wired lamplighters from \cite{CWL}, and \cite{Yang1} are linear, so finding a non-linear one or proving that they are all linear will give great insight into their structure. 
\end{enumerate}

\bibliography{references} 
\bibliographystyle{ieeetr}

\end{document}